\documentclass[12pt]{article}
\usepackage{color}
\definecolor{darkblue}{rgb}{0.00,0.25,0.50}
\usepackage[colorlinks,filecolor=blue,citecolor=darkblue]{hyperref}

\usepackage{amssymb,amsmath, amsfonts, amsthm}
\usepackage{url}
\usepackage{amscd, color}

\newtheorem{thrm}{Theorem}

\newtheorem{cor}[thrm]{Corollary}
\theoremstyle{definition}

\usepackage[english]{babel}
\begin{document}

\thispagestyle{empty}

\begin{center}
\textbf{APPROXIMATION OF CLASSES OF ANALYTIC FUNCTIONS BY DE LA VALL\'{E}E POUSSIN SUMS IN
UNIFORM METRIC$^\ast$\footnote{$^\ast\,$Supported in part by the Ukrainian Foundation for
Basic Research (project no. $\Phi$35/001)}}
\end{center}
\vskip0.5cm
\begin{center}
A.S. SERDYUK, Ie.Yu. OVSII and A.P. MUSIENKO
\end{center}
\vskip0.5cm

\abstract{In this paper asymptotic equalities are found for the least upper bounds of
deviations in the uniform metric of de la Vall\'{e}e Poussin sums on classes of $2\pi
$-periodic $(\psi ,\beta )$-differentiable functions admitting an analytic continuation into
the given strip of the complex plane. As a consequence, asymptotic equalities are obtained
on classes of convolutions of periodic functions generated by the Neumann kernel and the
polyharmonic Poisson kernel.}

\vskip0.2cm {\emph{MSC 2010}: 41A30 \vskip1cm

Let $L_s$, $1\leqslant s< \infty,$ be the space of $s$th power summable $2\pi $-periodic
functions $f$ with the norm $\|f\|_{s}:=\|f\|_{L_s}=\bigg(\int\limits_{0}^{2\pi
}|f(t)|^s\,dt\bigg)^{1/s},$ let $L_\infty$ be the space of measurable essentially bounded
$2\pi $-periodic functions $f$ with the norm $\|f\|_{\infty}:=\|f\|_{L_\infty}=\mathop{\rm
ess\,sup}\limits_{t\ }|f(t)|$ and let $C$ be the space of continuous $2\pi $-periodic
functions $f$ with the norm $\|f\|_C=\max\limits_{t}|f(t)|.$

Suppose that $f\in L_1$ and
    $$S[f]:=\frac{a_0(f)}{2}+\sum\limits_{
    k=1}^{\infty}(a_k(f)\cos kx+b_k(f)\sin kx)$$
is the Fourier series of $f$. If a sequence of real numbers $\psi (k)$, $k\in \mathbb{N}$
and a real number $\beta$ ($\beta \in \mathbb{R}$) are such that there exists a function
$\varphi \in L_1$ with Fourier series
    $$S[\varphi ]=\sum\limits_{
    k=1}^{\infty}\frac{1}{\psi (k)}\bigg(a_k(f)\cos\Big(kx+\frac{\beta\pi }{2}\Big)+
    b_k(f)\sin\Big(kx+\frac{\beta\pi }{2}\Big)\bigg),$$
then this function $\varphi $ is called (see \cite[p. 120]{Stepanets_VSP}) the $(\psi
,\beta)$-derivative of the function $f(\cdot)$ and is denoted by $f^\psi _{\beta}(\cdot).$
If $f^\psi _{\beta}\in \mathfrak{N}\subset L_1,$ then we write  $f\in L^\psi
_{\beta}\mathfrak{N}.$ Moreover, we set $C^\psi _{\beta}\mathfrak{N}=C\cap L^\psi
_{\beta}\mathfrak{N}.$

By $D_q$ we denote the set of sequences $\psi (k)> 0$, $k\in \mathbb{N},$ such that
\begin{equation}\label{24.05.11-18:27:57}
    \lim\limits_{k\to\infty}{\frac{\psi (k+1)}{\psi (k)}=q},\ \ \ q\in(0,1).
\end{equation}
It is known \cite[p. 130]{Stepanets_VSP} that the class $C^\psi _{\beta}\mathfrak{N}$ with
$\psi \in D_q$ consists of $2\pi $-periodic functions that admit an analytic continuation
into the strip $|\text{Im}\ z|\leqslant \ln {1/q}$ of the complex plane.

As follows from proposition 8.3 \cite[p. 127]{Stepanets_VSP}, if $\psi \in D_q$,
$q\in(0,1),$ $\beta \in \mathbb{R}$ and $\mathfrak{N}\subset L_{s},$ $1\leqslant s\leqslant
\infty,$ then $C^\psi _{\beta}\mathfrak{N}$ is the class of functions $f(x)$ representable
at each point $x\in\mathbb{R}$ by the equality
    \begin{equation}\label{1.06.11-15:14:39}
    f(x)=\frac{a_0(f)}{2}+\frac{1}{\pi }\int\limits_{
    -\pi }^{\pi }f^\psi _{\beta}(x-t)\Psi_{\beta}(t)\,dt,
    \end{equation}
where
    \begin{equation}\label{21.06.11-14:33:05}
    \Psi_{\beta}(t)=\sum\limits_{k=1}^{\infty}\psi (k)\cos\Big(k t-\frac{\beta\pi }{2}\Big).
    \end{equation}
An example of the class $C^\psi _{\beta }\mathfrak{N}$ for which $\psi \in D_q$,
$q\in(0,1)$, is the class of Poisson integrals, i.e. a class consisting of functions of the
form
    \begin{equation}\label{26.09.11-13:26:03}
    f(x)=A_0+\frac{
    1}{\pi }\int\limits_{-\pi }^{\pi }\varphi (x-t)P_{q,\beta }(t)\,dt,
    \ \ A_0\in\mathbb{R},\ \ \varphi \in\mathfrak{N},
    \end{equation}
where
    $$P_{q,\beta }(t)=
    \sum\limits_{k=1}^{\infty}q^k\cos \Big(kt-\frac{\beta \pi }{2}\Big),\ \
    q\in (0,1),\ \ \beta \in \mathbb{R},$$
is the Poisson kernel with parameters $q$ and $\beta .$ In this case the class $C^\psi
_\beta \mathfrak{N}$ we will denote by $C^q _\beta \mathfrak{N}.$

In the current paper we take as $\mathfrak{N}$ the sets
    $$U_s^0=\{\varphi \in L_s:\ \ \|\varphi \|_{s}\leqslant 1,\ \ \varphi \perp 1\},\ \ 1\leqslant s\leqslant \infty,$$
and
    $$H_\omega =\{\varphi \in C:\ \ \omega (\varphi ;t)\leqslant \omega (t),\ t\geqslant 0\},$$
where $\omega (\varphi ;t)$ is the modulus of continuity of $\varphi $ and $\omega (t)$ is a
fixed majorant of the modulus of continuity type. In what follows, we use the notation:
    $$C^\psi _{\beta,s}=C^\psi _{\beta}U_s^0,\ \ \
    C^q _{\beta,s}=C^q _{\beta }U_s^0.$$

Denote by $V_{n,p}(f;\cdot)$ the de la Vall\'{e}e Poussin sums \cite{Vallee Poussin} of the
function $f\in L_1:$
    \begin{equation}\label{19.07.11-15:09:23}
    V_{n,p}(f;x)=\frac{1}{p}\sum\limits_{k=n-p}^{n-1}S_k(f;x),
    \end{equation}
where $S_k(f;x)$ is the $k$th partial sum of the Fourier series of $f$, and $p=p(n)$ is a
given natural parameter, $p\leqslant n.$

The aim of the present work is to obtain the asymptotic equalities as \mbox{$n-p\to\infty$}
for the quantity
\begin{equation}\label{30.05.11-16:06:08}
    \mathcal{E}(C^\psi _{\beta }\mathfrak{N};V_{n,p})=\sup\limits_{
    f\in C^\psi _{\beta }\mathfrak{N}}{\|f(\cdot)-V_{n,p}(f;\cdot)\|_C,}
\end{equation}
where $\psi \in D_q,$ $q\in(0,1),$ and $\mathfrak{N}=U_s^0,$ $1\leqslant s\leqslant \infty,$
or $\mathfrak{N}=H_\omega .$

This paper is nearly related to works \cite{RUKASOV-2003}, \cite{Rukasov_Chaichenko},
\cite{Serdyuk V-P}, \cite{Serdyuk 2005}, \cite{Serdyuk 2010}, \cite{Serdyuk_Ovsii},
\cite{Serdyuk_Ovsii_Zb} and \cite{STEANETS_SERDYUK}. In \cite{Serdyuk_Ovsii_Zb} the
asymptotic equalities were obtained for $\mathcal{E}(C^\psi _{\beta,s };V_{n,p})$,
$1\leqslant s\leqslant \infty$ and $\mathcal{E}(C^\psi _{\beta}H_\omega ;V_{n,p})$ in the
case where the sequence $\psi (k)$, that defines the classes, satisfies the condition
    $$\lim\limits_{k\to\infty}{\frac{\psi (k+1)}{\psi (k)}}=0\ \ \ (\psi \in D_0).$$
This restriction on $\psi $ implies $C^\psi _{\beta,s }$ and $C^\psi _{\beta}H_\omega$ are
the classes of entire functions. The case $\psi \in D_q$, $q\in(0,1)$ also hasn't been
omitted. The solution of the problem under consideration  for $\psi \in D_q$, $q\in(0,1),$
and $p=1$ ($V_{n,1}(f;\cdot)=S_{n-1}(f;\cdot)$) was found in \cite{Serdyuk 2005} and
\cite{STEANETS_SERDYUK}. The main idea of paper \cite{STEANETS_SERDYUK} (see, also,
\cite[Chapt. 5, Sect. 20]{Stepanets_VSP}) consists of reduction of the problem of obtaining
asymptotic equalities for $\mathcal{E}(C^\psi _\beta \mathfrak{N};S_{n-1})$ to solving an
analogous problem for the quantity $\mathcal{E}(C^q _\beta \mathfrak{N};S_{n-1})$ by means
of the next equalities:
\begin{equation}\label{4.08.11-16:05:59}
    \mathcal{E}(C^\psi _{\beta,s};S_{n-1})=\psi(n)\bigg(q^{-n}\mathcal{E}(C^q_{\beta ,s};S_{n-1}
    )+O(1)\frac{\varepsilon_n}{(1-q)^2}\bigg),\ \ 1\leqslant s\leqslant \infty,
\end{equation}
\begin{equation}\label{4.08.11-16:18:11}
    \mathcal{E}(C^\psi _{\beta}H_\omega ;S_{n-1})=\psi(n)\bigg(q^{-n}\mathcal{E}(C^q_{\beta}H_\omega ;S_{n-1}
    )+O(1)\frac{\varepsilon_n\omega (1/n)}{(1-q)^2}\bigg),
\end{equation}
where $\varepsilon _{n}:=\sup\limits_{k\geqslant n}{\big|\frac{ \psi
(k+1)}{\psi(k)}-q\big|,}$ and $O(1)$ are the quantities uniformly bounded in $n,$ $s,$ $q,$
$\psi (k)$ and $\beta .$ Since the asymptotic equalities for $\mathcal{E}(C^q_{\beta
,s};S_{n-1})$ and $\mathcal{E}(C^q_{\beta}H_\omega ;S_{n-1})$ are known (see, for example,
\cite[p. 295, 310]{Stepanets_VSP}, \cite[p. 1278]{Serdyuk 2005}), formulas
(\ref{4.08.11-16:05:59}) and (\ref{4.08.11-16:18:11}) allow us to obtain the asymptotic
equalities for $\mathcal{E}(C^\psi _{\beta,s};S_{n-1})$ and $\mathcal{E}(C^\psi
_{\beta}H_\omega ;S_{n-1})$, respectively, with arbitrary $\beta \in\mathbb{R}$ and $\psi
\in D_q,$ $q\in(0,1).$

As for the general case $p=1,2,\ldots,n$, the analogs of (\ref{4.08.11-16:05:59}) (with
$s=\infty$) and (\ref{4.08.11-16:18:11}) were derived in \cite{RUKASOV-2003} and have the
form
\begin{equation}\label{4.08.11-16:40:18}
    \mathcal{E}(C^\psi _{\beta ,\infty};V_{n,p})=\psi (n-p+1)\bigg(
    \frac{\mathcal{E}(C^q _{\beta ,\infty};V_{n,p})}{q^{n-p+1}}+
    O(1)\frac{\varepsilon_{n-p+1}}{(1-q)^4}\bigg),
\end{equation}
\begin{equation}\label{4.08.11-16:43:08}
    \mathcal{E}(C^\psi _\beta H_\omega ;V_{n,p})=\psi (n-p+1)\bigg(
    \frac{\mathcal{E}(C^q_\beta H_\omega ;V_{n,p})}{q^{n-p+1}}+
    O(1)\omega \Big(\frac{1}{n-p+1}\Big)\frac{\varepsilon _{n-p+1}}{(1-q)^4}\bigg),
\end{equation}
where $\psi \in  D_q$, $q\in(0,1),$ $\beta \in\mathbb{R},$
\begin{equation}\label{4.08.11-16:46:57}
    \varepsilon _{n-p+1}:=\sup\limits_{k\geqslant n-p+1}{\Big|\frac{ \psi (k+1)}{\psi
(k)}-q\Big|},
\end{equation}
$\omega (t)$ is an arbitrary modulus of continuity and $O(1)$ are the quantities uniformly
bounded in $n,$ $p,$ $q,$ $\psi $ and $\beta .$

By using the known asymptotic equalities as $n-p\to\infty$ of the quantities
$\mathcal{E}(C^q _{\beta ,\infty};V_{n,p})$ and $\mathcal{E}(C^q _\beta H_\omega ;V_{n,p})$
(see \cite{RUKASOV-2003} and \cite{Rukasov_Chaichenko}), \mbox{V. I. Rukasov} obtained from
(\ref{4.08.11-16:40:18}) and (\ref{4.08.11-16:43:08}) the next formulas that make up the
main result of paper \cite{RUKASOV-2003}:
\begin{equation}\label{4.08.11-16:58:12}
    \mathcal{E}(C^\psi _{\beta ,\infty};V_{n,p})=\frac{
    \psi (n-p+1)}{p}\bigg(\frac{4}{\pi (1-q^2)}$$
    $$+O(1)\Big(
    \frac{q^{p-1}}{(1-q^2)}+\frac{1}{(1-q)^3(n-p)}+\frac{p\varepsilon
    _{n-p}}{(1-q)^4}\Big)\bigg),
\end{equation}
\begin{equation}\label{4.08.11-17:06:21}
    \mathcal{E}(C^\psi _\beta H_\omega ;V_{n,p})=\frac{\psi (n-p+1)}{p}\bigg(
    \frac{2\theta_\omega }{\pi (1-q^2)}\int\limits_{0}^{\pi /2}\omega \Big(
    \frac{2t}{n-p}\Big)\sin t\,dt$$
    $$+O(1)\omega \Big(\frac{1}{n-p}\Big)\Big(
    \frac{q^{p-1}}{(1-q^2)}+\frac{1}{(1-q)^3(n-p)}+\frac{p\varepsilon
    _{n-p}}{(1-q)^4}\Big)\bigg),
\end{equation}
where $\psi \in D_q,$ $q\in(0,1),$ $\beta \in\mathbb{R}$, $\varepsilon
_{n-p}=\sup\limits_{k\geqslant n-p}{\big|\frac{ \psi (k+1)}{\psi (k)}-q\big|},$
$\theta_\omega \in[1/2,1]$ $(\theta_\omega =1$ if $\omega (t)$ is a convex (upwards) modulus
of continuity) and the quantities $O(1)$ are uniformly bounded in $n$, $p,$ $q,$ $\psi $ and
$\beta .$

Formula (\ref{4.08.11-16:58:12}), as well as formula (\ref{4.08.11-17:06:21}) in the case of
convexity of modulus of continuity $\omega (t)$, is an asymptotic equality as $n-p\to\infty$
only if the additional conditions
\begin{equation}\label{4.08.11-17:16:56}
    \lim\limits_{n\to\infty}{p}=\infty,
\end{equation}
\begin{equation}\label{4.08.11-17:17:19}
    \lim\limits_{n\to\infty}{p\varepsilon _{n-p}}=0,
\end{equation}
are fulfilled.

In the present work we have been able to do away with restrictions (\ref{4.08.11-17:16:56})
and (\ref{4.08.11-17:17:19}); this means that the strong asymptotic as $n-p\to\infty$ of
$\mathcal{E}(C^\psi _{\beta ,s};V_{n,p})$ and $\mathcal{E}(C^\psi _\beta H_\omega ;V_{n,p})$
is obtained for arbitrary $\psi \in D_q,$ $q\in(0,1),$ $1\leqslant s\leqslant \infty,$
$\beta \in\mathbb{R}$ even in the case where at least one of (\ref{4.08.11-17:16:56}) or
(\ref{4.08.11-17:17:19}) isn't carried out. It's essential to note that reasoning from
relations (\ref{4.08.11-16:40:18}) and (\ref{4.08.11-16:43:08}), restrictions
(\ref{4.08.11-17:16:56}) and (\ref{4.08.11-17:17:19}) can't be removed in principle. Thus,
for the final solution of our problem, it needs to improve formulas (\ref{4.08.11-16:40:18})
and (\ref{4.08.11-16:43:08}), which we shall do finding more refined estimates of the
remainder terms with subsequent generalisation of (\ref{4.08.11-16:40:18}) to the case of
arbitrary $s\in[1,\infty].$ The sought-for relations are provided by the following
assertion, which plays a key role in this paper.

\begin{thrm}\label{t1}
Let $\psi\in D_q$, $q\in(0,1),$ $1\leqslant s\leqslant \infty,$ $n,p\in\mathbb{N},$
$p\leqslant n,$ $\beta \in\mathbb{R}$ and let $\omega (t)$ be an arbitrary modulus of
continuity. Then, as $n-p\to\infty$,
\begin{equation}\label{23.05.11-17:21:26}
    \mathcal{E}(C^\psi _{\beta,s};V_{n,p})=\frac{\psi (n-p+1)}{p}\bigg(
    \frac{p\mathcal{E}(C^q _{\beta,s};V_{n,p})}{q^{n-p+1}}+
    O(1)\frac{\varepsilon _{n-p+1}}{(1-q)^2}\min\Big\{p, \frac{1}{1-q}\Big\}\bigg),
\end{equation}
\begin{equation}\label{23.05.11-17:23:37}
    \mathcal{E}(C^\psi _{\beta}H_\omega ;V_{n,p})=\frac{\psi (n-p+1)}{p}\bigg(
    \frac{p\mathcal{E}(C^q _{\beta}H_\omega ;V_{n,p})}{q^{n-p+1}}$$
    $$+O(1)\omega
    \Big(\frac{1}{n-p+1}\Big)\frac{\varepsilon _{n-p+1}}{(1-q)^2}\min\Big\{p, \frac{1}{1-q}\Big\}\bigg),
\end{equation}
where $\varepsilon _{n-p+1}$ is defined by $(\ref{4.08.11-16:46:57})$ and $O(1)$ are the
quantities uniformly bounded in $n,$ $p$, $q$, $s,$ $\psi, $ $\beta$ and $\omega .$
\end{thrm}

\begin{proof} Let  $f\in C^{\psi}_{\beta}\mathfrak{N},$ $\psi \in D_q$,
$q\in(0,1)$ and $\mathfrak{N}=U_s^0,$ \mbox{$1\leqslant s\leqslant \infty$,} or
$\mathfrak{N}=H_\omega .$ By (\ref{1.06.11-15:14:39}) and (\ref{19.07.11-15:09:23}), the
deviation
    $$\rho _{n,p}(f;x):=f(x)-V_{n,p}(f;x),$$
satisfies at each point $x\in\mathbb{R}$ the equality
\begin{equation}\label{23.05.11-14:03:05}
    \rho _{n,p}(f;x)=\frac{1}{\pi }\int\limits_{-\pi }^{\pi }f^\psi _\beta  (x-t)\sum\limits_{
    k=n-p+1}^{\infty}\tau _{n,p}(k)\psi (k)\cos\Big(kt-\frac{\beta\pi }{2}\Big)\,dt,\ \ f^\psi _\beta
    \in \mathfrak{N},
\end{equation}
where
\begin{equation}\label{7.06.11-16:49:50}
\tau _{n,p}(k)=
  \begin{cases}
    1-\frac{n-k}{p}, & n-p+1\leqslant k\leqslant n-1,\\
    1, & k\geqslant n.
  \end{cases}
\end{equation}
Setting
\begin{equation}\label{23.05.11-14:08:26}
    r_{n,p}(t):=\sum\limits_{k=n-p+2}^{\infty}\tau _{n,p}(k)
    \bigg(\frac{
    \psi (k)}{\psi (n-p+1)}-\frac{q^{k}}{q^{n-p+1}}\bigg)\cos\Big(kt-\frac{\beta\pi }{2
    }\Big),
\end{equation}
we rewrite (\ref{23.05.11-14:03:05}) thus:
\begin{equation}\label{23.05.11-14:12:28}
    \rho _{n,p}(f;x)=\psi (n-p+1)\bigg(
    \frac{q^{p-n-1}}{\pi }\int\limits_{-\pi }^{\pi }
    f^\psi _\beta  (x-t)\sum\limits_{k=n-p+1}^{\infty}\tau _{n,p}(k)q^k\cos\Big(kt-\frac{\beta\pi }{2
    }\Big)\,dt$$
    $$+\frac{1}{\pi }\int\limits_{-\pi }^{\pi }f^\psi _\beta (x-t)r_{n,p}(t)\,dt\bigg).
\end{equation}
Since, by virtue of (\ref{23.05.11-14:03:05}),
    $$\mathcal{E}(C^q _{\beta}\mathfrak{N};V_{n,p}):=
    \sup\limits_{f\in C^q _{\beta}\mathfrak{N}}{\|\rho _{n,p}(f;\cdot)\|}_C$$
    \begin{equation}\label{27.09.11-12:43:31}
    =\sup\limits_{\varphi \in \mathfrak{N}}\bigg\|\frac{1}{\pi }\int\limits_{-\pi }^{\pi }
    \varphi (\cdot-t)\sum\limits_{k=n-p+1}^{\infty}\tau _{n,p}(k)q^k\cos\Big(kt-\frac{\beta\pi }{2
    }\Big)\,dt\bigg\|_C,\ \ \mathfrak{N}\subset L_1,
    \end{equation}
it follows from (\ref{23.05.11-14:12:28}) and (\ref{27.09.11-12:43:31}) that
\begin{equation}\label{23.05.11-14:23:30}
    \mathcal{E}(C^\psi _{\beta}\mathfrak{N};V_{n,p})=\psi (n-p+1)\bigg(
    \frac{\mathcal{E}(C^q _{\beta}\mathfrak{N};V_{n,p})}{q^{n-p+1}}+O(1)
    \sup\limits_{\varphi \in \mathfrak{N}}{}\bigg\|\int\limits_{-\pi }^{\pi }\varphi
    (\cdot-t)r_{n,p}(t)\,dt\bigg\|_C
    \bigg).
\end{equation}
If $\mathfrak{N}=U_s^0$, $1\leqslant s\leqslant \infty,$  we get from the H\"{o}lder
inequality (see, e.g., \cite[p. 410]{KORNEICHUK_Exact_Constant})
\begin{equation}\label{23.05.11-15:25:13}
    \sup\limits_{\varphi \in U_s^0}\bigg|
    \int\limits_{-\pi }^{\pi }\varphi (x-t)r_{n,p}(t)\,dt\bigg|\leqslant \|r_{n,p}(\cdot)\|_{s'},\ \
    \frac{1}{s}+\frac{1}{s'}=1.
\end{equation}
If $\mathfrak{N}=H_\omega $, then considering that the function $r_{n,p}(t)$ (see
(\ref{23.05.11-14:08:26})) and a random trigonometric polynomial $T_{n-p}(\cdot)$ of order
not more than $n-p$ are orthogonal, we can write
    $$\sup\limits_{\varphi \in H_\omega }{}\bigg|
    \int\limits_{-\pi }^{\pi }\varphi (x-t)r_{n,p}(t)\,dt\bigg|=
    \sup\limits_{\varphi \in H_\omega }{}\bigg|
    \int\limits_{-\pi }^{\pi }(\varphi (x-t)-T_{n-p}(x-t))r_{n,p}(t)\,dt\bigg|$$
    \begin{equation}\label{23.05.11-15:33:23}
    \leqslant \sup\limits_{\varphi \in H_\omega }\|\varphi (\cdot)-
    T_{n-p}(\cdot)\|_C\|r_{n,p}(\cdot)\|_1.
    \end{equation}
Let $T_{n-p}^*(\cdot)$ be the polynomial of best uniform approximation of the function
$\varphi \in H_\omega $ by means of trigonometric polynomials of order $\leqslant n-p:$
    $$\|\varphi(\cdot)-T_{n-p}^*(\cdot) \|_C=\inf\limits_{
    T_{n-p}}{\|\varphi (\cdot)-T_{n-p}(\cdot)\|_C}=:E_{n-p+1}(\varphi ).$$
Then, by choosing $T_{n-p}^*(\cdot)$ as the polynomial $T_{n-p}(\cdot)$ in
(\ref{23.05.11-15:33:23}) and using the well-known Jackson inequality (see, for example,
\cite[p. 266]{KORNEICHUK_Exact_Constant})
    $$E_{n-p+1}(\varphi )\leqslant K\omega \Big(\varphi ,\frac{1}{n-p+1}\Big),\ \ \ K=\text{const},$$
we get from (\ref{23.05.11-15:33:23}) the estimate
\begin{equation}\label{23.05.11-15:48:07}
    \sup\limits_{\varphi \in H_\omega }{}\bigg|
    \int\limits_{-\pi }^{\pi }\varphi (x-t)r_{n,p}(t)\,dt\bigg|=O(1)
    \omega \Big(\frac{1}{n-p+1}\Big)\|r_{n,p}(\cdot)\|_1.
\end{equation}
We show that
\begin{equation}\label{7.06.11-17:08:12}
    r_{n,p}(t)=O(1)\frac{\varepsilon _{n-p+1}}{(1-q)^2}\min\limits{\bigg\{1,\frac{1}{
    p(1-q)}\bigg\}},\ \ n-p\to\infty,
\end{equation}
where $O(1)$ is the quantity uniformly bounded in $t,$ $n$, $p$,  $q$, $\psi $ and $\beta .$

To do this we first rewrite (\ref{23.05.11-14:08:26}) in the form
    $$r_{n,p}(t)=\sum\limits_{k=n-p+2}^{\infty}\tau _{n,p}(k)
    \bigg(\prod\limits_{l=0}^{k-n+p-2}\frac{
    \psi (n-p+2+l)}{\psi (n-p+1+l)}-\frac{q^{k}}{q^{n-p+1}}\bigg)\cos\Big(kt-\frac{\beta\pi }{2
    }\Big).$$
Since $\tau _{n,p}(k)>0$,
    $$
    |r_{n,p}(t)|\leqslant \sum\limits_{k=1}^{\infty}\tau _{n,p}(n-p+1+k)
    \bigg|\prod\limits_{l=0}^{k-1}\frac{
    \psi (n-p+2+l)}{\psi (n-p+1+l)}-q^k\bigg|.
    $$
By the estimate
    \begin{equation}\label{6.06.11-15:22:48}
    \bigg|\prod\limits_{l=0}^{k-1}\frac{
    \psi (m+l+1)}{\psi (m+l)}-q^k\bigg|\leqslant(q+\varepsilon _m)^k-q^k,\ \ m\in\mathbb{N},
    \end{equation}
proved in \cite[p. 438]{STEANETS_SERDYUK}, this implies that
    \begin{equation}\label{6.06.11-15:22:31}
    |r_{n,p}(t)|\leqslant \sum\limits_{k=1}^{\infty}\tau _{n,p}(n-p+1+k)\Big((q+\varepsilon _{n-p+1})^k-q^k\Big).
    \end{equation}
The sequence $\varepsilon _{m}$ tends monotonically to zero. Hence, for sufficiently large
$n-p+1,$
    \begin{equation}\label{27.10.11-10:56:45}
    \varepsilon _{n-p+1}<\frac{1-q}{2}.
    \end{equation}
Therefore, taking into account the fact that $\tau _{n,p}(k)\leqslant 1$ and using the
formula
\begin{equation}\label{21.06.11-15:40:12}
    \sum\limits_{k=1}^{\infty}x^k=\frac{x}{1-x},\ \ 0<x<1,
\end{equation}
from (\ref{6.06.11-15:22:31}) we have
    \begin{equation}\label{18.05.11-11:44:42}
    |r_{n,p}(t)|\leqslant \frac{\varepsilon _{n-p+1}}{(1-q)(1-q-\varepsilon _{n-p+1})}
    <2\frac{\varepsilon _{n-p+1}}{(1-q)^2}.
    \end{equation}
On the other hand, from (\ref{6.06.11-15:22:31}) and (\ref{7.06.11-16:49:50}) we obtain
    $$
    |r_{n,p}(t)|\leqslant \sum\limits_{k=1}^{\infty}\tau _{n,p}(n-p+1+k)
    \Big((q+\varepsilon _{n-p+1})^k-q^k\Big)$$
    $$=\sum\limits_{k=1}^{p-2}\frac{k+1}{p}\Big((q+\varepsilon _{n-p+1})^k-q^k\Big)+
    \sum\limits_{k=p-1}^{\infty}\Big((q+\varepsilon _{n-p+1})^k-q^k\Big)$$
\begin{equation}\label{6.06.11-16:06:34}
<\sum\limits_{k=1}^{\infty}\frac{k+1}{p}\Big((q+\varepsilon _{n-p+1})^k-q^k\Big) .
\end{equation}
Estimate (\ref{6.06.11-16:06:34}) together with the equality
    $$\sum\limits_{k=1}^{\infty}kx^k=\frac{x}{(1-x)^2},\ \ 0<x<1$$
and (\ref{27.10.11-10:56:45}) imply that for sufficiently large $n-p+1$
$$
    |r_{n,p}(t)|<2\frac{\varepsilon _{n-p+1}}{p}\frac{(1-q-\frac{\varepsilon _{n-p+1}}{2}
    )}{(1-q)^2(1-q-\varepsilon _{n-p+1})^2}<\frac{8}{p}\frac{\varepsilon _{n-p+1}}{(1-q)^3}.
$$
In combination with (\ref{18.05.11-11:44:42}) this yields estimate (\ref{7.06.11-17:08:12}).

Gathering together (\ref{23.05.11-14:23:30}), (\ref{23.05.11-15:25:13}),
(\ref{23.05.11-15:48:07}) and (\ref{7.06.11-17:08:12}) we obtain Theorem \ref{t1}.
\end{proof}

The quantity $pq^{-(n-p+1)}\mathcal{E}(C^q_{\beta ,s};V_{n,p})$ is bounded above and below
by some positive numbers, possibly depending only on $q$ and $s$. Indeed, on the strength of
(\ref{27.09.11-12:43:31}),
    $$\mathcal{E}(C^q_{\beta ,s};V_{n,p})=
    \sup\limits_{\varphi \in U_s^0}\bigg\|\frac{1}{\pi }\int\limits_{-\pi }^{\pi }
    \varphi (\cdot-t)\sum\limits_{k=n-p+1}^{\infty}\tau _{n,p}(k)
    q^k\cos\Big(kt-\frac{\beta\pi }{2
    }\Big)\,dt\bigg\|_C$$
    $$\leqslant C_s^{(1)}\bigg\|\sum\limits_{k=n-p+1}^{\infty}\tau _{n,p}(k)
    q^k\cos\Big(kt-\frac{\beta\pi }{2
    }\Big)\bigg\|_{s'}.$$
Since
    \begin{equation}\label{27.09.11-10:13:20}
    \sum\limits_{k=n-p+1}^{\infty}\tau _{n,p}(k)q^k<\frac{1}{p}\sum\limits_{k=1}^{\infty}kq^{k+n-p}=
    \frac{1}{p}\frac{q^{n-p+1}}{(1-q)^2}
    \end{equation}
by (\ref{7.06.11-16:49:50}), we conclude that
    $$pq^{-(n-p+1)}\mathcal{E}(C^q_{\beta ,s};V_{n,p})\leqslant \frac{C_s^{(1)}}{(1-q)^2}.$$

To find a lower estimate of the quantity $pq^{-(n-p+1)}\mathcal{E}(C^q_{\beta ,s};V_{n,p}),$
it is sufficient to consider the function
    $$f_{n-p+1}(x)=q^{n-p+1}\|\sin t\|_s^{-1}\sin\Big((n-p+1)x+\frac{\beta \pi }{2}\Big).$$
The function $f_{n-p+1}(x)$ belongs to $C^q_{\beta ,s}$ and so
    $$pq^{-(n-p+1)}\mathcal{E}(C^q_{\beta ,s};V_{n,p})\geqslant pq^{-(n-p+1)}
    \|\rho_{n,p} (f_{n-p+1};\cdot)\|_C=\frac{\|\sin t\|_C}{\|\sin t\|_s}=C^{(2)}_s>0.$$

Thus,
\begin{equation}\label{26.09.11-13:47:23}
    C^{(1)}_s\leqslant pq^{-(n-p+1)}\mathcal{E}(C^q_{\beta ,s};V_{n,p})\leqslant C^{(2)}_s\frac{
    1}{(1-q)^2}, \ \ C_s^{(i)}>0,\ \ i=1,2.
\end{equation}

An analogous estimate also holds for $pq^{-(n-p+1)}\mathcal{E}(C^q_{\beta}H_\omega
;V_{n,p}):$
\begin{equation}\label{26.09.11-13:51:46}
    C^{(1)}_s\omega \Big(\frac{1}{n-p+1}\Big)\leqslant
    pq^{-(n-p+1)}\mathcal{E}(C^q_{\beta}H_\omega ;V_{n,p})\leqslant
    \frac{C^{(2)}_s}{(1-q)^2}\omega \Big(\frac{1}{n-p+1}\Big),
\end{equation}
where $C^{(i)}_s>0,$ $i=1,2.$

Indeed, since the function $\sum\limits_{k=n-p+1}^{\infty}\tau
_{n,p}(k)q^k\cos\Big(kt-\frac{\beta\pi }{2}\Big)$ is orthogonal to every trigonometric
polynomial $T_{n-p}(\cdot)$ of order $\leqslant n-p$, from (\ref{27.09.11-12:43:31}) we have
    \begin{equation}\label{27.09.11-10:06:04}
    \mathcal{E}(C^q_\beta H_\omega ;V_{n,p})\leqslant C_s^{(1)}
    \sup\limits_{\varphi \in H_\omega }{\|\varphi (\cdot)-T_{n-p}(\cdot)\|_C}\bigg\|\sum\limits_{k=n-p+1}^{\infty}\tau _{n,p}(k)
    q^k\cos\Big(kt-\frac{\beta\pi }{2
    }\Big)\bigg\|_{1}.
    \end{equation}
Choosing the polynomial of best approximation of the function $\varphi\in H_\omega $ as
$T_{n-p}(\cdot)$ in (\ref{27.09.11-10:06:04}) and applying the Jackson inequality and
(\ref{27.09.11-10:13:20}), we obtain
\begin{equation}\label{27.09.11-11:26:09}
    pq^{-(n-p+1)}\mathcal{E}(C^q_{\beta}H_\omega ;V_{n,p})\leqslant
    \frac{C^{(1)}_s}{(1-q)^2}\omega \Big(\frac{1}{n-p+1}\Big).
\end{equation}
On the other hand,
    \begin{equation}\label{27.09.11-11:10:21}
    pq^{-(n-p+1)}\mathcal{E}(C^q_{\beta}H_\omega ;V_{n,p})\geqslant
    q^{-(n-p+1)}\mathcal{E}(C^q_{\beta}H_\omega ;V_{n,p})\geqslant q^{-(n-p+1)}
    E_{n-p+1}(C^q_\beta H_\omega ),
    \end{equation}
where $E_{n-p+1}(C^q_\beta H_\omega ) = \sup\limits_{f\in C^q_\beta H_\omega
}{\inf\limits_{T_{n-p}}{\|f(\cdot)-T_{n-p}(\cdot)\|_C}}.$ As follows from formula (8) in
\cite{Serdyuk_Sokolenko_JAT}, the next estimate holds for the quantity $E_{n-p+1}(C^q_\beta
H_\omega ):$
\begin{equation}\label{27.09.11-11:18:47}
    E_{n-p+1}(C^q_\beta H_\omega)\geqslant C_s^{(2)}q^{n-p+1}\omega \Big(\frac{1}{
    n-p+1}\Big).
\end{equation}
Comparing (\ref{27.09.11-11:26:09})--(\ref{27.09.11-11:18:47}), we get
 (\ref{26.09.11-13:51:46}).

Since $\varepsilon _{n-p+1}\to 0$ as $n-p\to\infty$, in view of (\ref{26.09.11-13:47:23})
and (\ref{26.09.11-13:51:46}) we conclude that in all cases where the asymptotic equalities
for $\mathcal{E}(C^q_{\beta ,s};V_{n,p})$ and $\mathcal{E}(C^q_{\beta}H_\omega ;V_{n,p})$
are known, relations (\ref{23.05.11-17:21:26}) and (\ref{23.05.11-17:23:37}) let us write
the analogous equalities for the quantities $\mathcal{E}(C^\psi _{\beta ,s};V_{n,p})$ and
$\mathcal{E}(C^\psi _{\beta}H_\omega ;V_{n,p})$, respectively, for any $\psi \in D_q,$
$q\in(0,1).$

This fact enables us to give some important corollaries from Theorem \ref{t1}. With this
aim, we cite one of the results from \cite[p. 1943]{Serdyuk 2010}, where it was shown that
for $q\in(0,1),$ $\beta \in\mathbb{R},$ \mbox{$1\leqslant s\leqslant \infty$} and
$n,p\in\mathbb{N}$, $p\leqslant n,$ the following asymptotic equality holds as
$n-p\to\infty:$
\begin{equation}\label{21.06.11-13:12:35}
    \mathcal{E}(C^q _{\beta,s};V_{n,p})=\frac{
    q^{n-p+1}}{p}\bigg(\frac{
    \|\cos t\|_{s'}}{\pi ^{1+1/s'}}K_{q,p}(s')+\frac{O(1)}{(n-p+1)(1-q)^{\sigma
    (s',p)}}\bigg),
\end{equation}
in which
    \begin{equation}\label{21.06.11-13:57:09}
    K_{q,p}(s'):=2^{-1/s'}\bigg\|\frac{
    \sqrt{1-2q^p\cos pt+q^{2p}}}{1-2q\cos t+q^2}\bigg\|_{s'},\ \ s'=\frac{s}{s-1},
    \end{equation}
    \begin{equation}\label{21.06.11-13:57:22}
    \sigma (s',p)=
  \begin{cases}
    1, & s'=1,\ \ p=1 , \\
    2, & 1<s'\leqslant \infty,\ \ p=1, \\
    3, & 1\leqslant s'\leqslant \infty,\ \ p\in\mathbb{N}\setminus \{1\},
  \end{cases}
    \end{equation}
and $O(1)$ is the quantity uniformly bounded in $n,$ $p,$ $q,$ $\beta $ and $s.$

For $s=\infty$ asymptotic equality (\ref{21.06.11-13:12:35}) was obtained in \cite{Serdyuk
V-P}.

Combining (\ref{23.05.11-17:21:26}) and (\ref{21.06.11-13:12:35}), we have.

\begin{thrm}\label{t2}
Let $\psi\in D_q$, $q\in(0,1),$ $1\leqslant s\leqslant \infty,$ $\beta \in \mathbb{R},$
$n,p\in\mathbb{N},$ $p\leqslant n.$ Then the following asymptotic equality holds as
$n-p\to\infty:$
\begin{equation}\label{21.06.11-13:54:59}
    \mathcal{E}(C^\psi _{\beta,s};V_{n,p})=
    \frac{\psi (n-p+1)}{p}\Bigg(\frac{
    \|\cos t\|_{s'}}{\pi ^{1+1/s'}}K_{q,p}(s')$$
    $$+O(1)\bigg(\frac{1}{(n-p+1)(1-q)^{\sigma
    (s',p)}}+\frac{\varepsilon _{n-p+1}}{(1-q)^2 }\min\Big\{p,\frac{1}{1-q}\Big\}\bigg)\Bigg),
\end{equation}
where $K_{q,p}(s')$ and $\sigma(s',p)$ are defined by $(\ref{21.06.11-13:57:09})$ and
$(\ref{21.06.11-13:57:22})$, respectively, $s'=\frac{s}{s-1},$ $\varepsilon
_{n-p+1}=\sup\limits_{k\geqslant n-p+1}{\big|\frac{ \psi (k+1)}{\psi (k)}-q\big|},$ and
$O(1)$ is the quantity uniformly bounded in $n,$ $p$, $q$, $s,$ $\psi $ and $\beta.$
\end{thrm}

Note that in the case where $p=1$ and $s\in[1,\infty],$ equality (\ref{21.06.11-13:54:59})
was established in \cite[p. 1289]{Serdyuk 2005}.

From the obvious relations
    $$1-q^p\leqslant \sqrt{1-2q^p\cos pt+q^{2p}}\leqslant 1+q^p$$
we can write that for $s=\infty$
    \begin{equation}\label{27.09.11-15:54:49}
    K_{q,p}(s')=K_{q,p}(1)=\int\limits_{0}^{\pi }\frac{
    \sqrt{1-2q^p\cos pt+q^{2p}}}{1-2q\cos t+q^2}\,dt=\frac{1}{1-q^2}(\pi +
    O(1)q^p).
    \end{equation}
Thus, from (\ref{21.06.11-13:54:59}) and (\ref{27.09.11-15:54:49}) we obtain the next
asymptotic equality as $n-p\to\infty$ and $p\to\infty:$
\begin{equation}\label{27.09.11-15:58:38}
    \mathcal{E}(C^\psi _{\beta ,\infty};V_{n,p})=
    \frac{\psi (n-p+1)}{p}\Bigg(
    \frac{4}{\pi (1-q^2)}$$
    $$+O(1)\bigg(
    \frac{q^p}{1-q}+\frac{1}{(n-p+1)(1-q)^{\sigma (1,p)}}+\frac{
    \varepsilon _{n-p+1}}{(1-q)^3}\bigg)\Bigg),
\end{equation}
where $\psi \in D_q$, $q\in(0,1),$ $\beta \in \mathbb{R}$, $\sigma (1,p)$ is defined by
(\ref{21.06.11-13:57:22}) and $O(1)$ is the quantity uniformly bounded in $n,p,q,\psi $ and
$\beta. $ Equality (\ref{27.09.11-15:58:38}) improves (\ref{4.08.11-16:58:12}) at the cost
of more precise estimate of the remainder term, it still remains asymptotic even though
restriction (\ref{4.08.11-17:17:19}) doesn't hold.

In the case of arbitrary $p=1,2,\ldots,n$ the behaviour of the constant $K_{q,p}(1)$ could
be inferred by the next identity, proved in \cite[p. 215]{Savchuks and Chaichenko}:
\begin{equation}\label{27.09.11-16:13:56}
    K_{q,p}(1)=2\frac{1-q^{2p}}{1-q^2}\textbf{K}(q^p),
\end{equation}
where $\textbf{K}(\rho )=\int\limits_{0}^{\pi /2}\frac{dt}{\sqrt{1-\rho ^2\sin^2 t}}$ is the
complete elliptic integral of the first kind. Taking (\ref{21.06.11-13:54:59}) and
(\ref{27.09.11-16:13:56}) together, we get that for any $\psi \in D_q$, $q\in(0,1),$ $\beta
\in \mathbb{R}$ and $n,p\in \mathbb{N}$ the asymptotic equality
\begin{equation}\label{27.09.11-16:19:51}
    \mathcal{E}(C^\psi _{\beta ,\infty};V_{n,p})=
    \frac{\psi (n-p+1)}{p}\Bigg(\frac{8}{\pi ^2}\frac{
    1-q^{2p}}{1-q^2}\textbf{K}(q^p)$$
    $$+O(1)\bigg(\frac{1}{(n-p+1)(1-q)^{\sigma (1,p)}}+\frac{
    \varepsilon _{n-p+1}}{(1-q)^2}\min\Big\{p,\frac{1}{1-q}\Big\}\bigg)\Bigg)
\end{equation}
is true as $n-p\to\infty$.

In the case $p=1$ equality (\ref{27.09.11-16:19:51}) was proved in \cite[p.
443]{STEANETS_SERDYUK}.

An analog of (\ref{27.09.11-16:19:51}) can be obtained for the class $C^\psi _\beta H_\omega
$ given by convex modulus of continuity $\omega (t).$ To this end, we use the following
equality (see \cite[p. 5]{Serdyuk_Ovsii})
\begin{equation}\label{28.09.11-11:05:42}
    \mathcal{E}(C^q_\beta H_\omega ;V_{n,p})=
    \frac{q^{n-p+1}}{p}\Bigg(
    \frac{4}{\pi ^2}\frac{1-q^{2p}}{1-q^2}\textbf{K}(q^p)\int\limits_{0}^{\pi /2}
    \omega \Big(\frac{2t}{n-p+1}\Big)\sin t\,dt
    $$
    $$+\frac{O(1)\omega (\pi )}{(1-q)^{\gamma   (p)}(n-p+1)}\Bigg), \ \ \ n-p\to\infty,
\end{equation}
valid for every $q\in(0,1)$, $\beta \in\mathbb{R}$ and every convex modulus of continuity
$\omega (t)$, in which
  \begin{equation}\label{28.09.11-11:12:43}
  \gamma (p)=\begin{cases}
    2, & p=1, \\
    3, & p=2,3,\ldots,n,
  \end{cases}
  \end{equation}
and the quantity $O(1)$ is uniformly bounded in $n$, $p$, $q$, $\beta $ and $\omega.$

Noting that (\ref{28.09.11-11:05:42}) is an asymptotic equality if and only if $\omega (t)$
satisfies the condition
\begin{equation}\label{28.09.11-11:14:59}
    \lim\limits_{t\to 0}{\frac{\omega (t)}{t}}=\infty,
\end{equation}
on the basis of (\ref{23.05.11-17:23:37}) and (\ref{28.09.11-11:05:42}) we arrive at the
following assertion.

\begin{thrm}\label{t3} Let $\psi\in D_q$, $q\in(0,1)$, \mbox{$n,p\in\mathbb{N}$}, $p\leqslant
n$ and let $\omega (t)$ be a convex modulus of continuity satisfying condition
$(\ref{28.09.11-11:14:59})$. Then, as \mbox{$n-p\to\infty,$}
\begin{equation}\label{28.09.11-11:28:32}
    \mathcal{E}(C^\psi_{\beta}H_\omega ;V_{n,p})=
    \frac{\psi (n-p+1)}{p}\Bigg(
    \frac{4}{\pi ^2}\frac{1-q^{2p}}{1-q^2}{\bf K}(q^p)\int\limits_{
    0}^{\pi /2}\omega \left(\frac{2t}{n-p+1}\right)\sin t\,dt
    $$
    $$+O(1)\bigg(\frac{\omega (\pi )}{(1-q)^{\gamma(p)}(n-p+1)}+
    \frac{\varepsilon _{n-p+1}}{(1-q)^2}\min\Big\{p, \frac{1}{1-q}\Big\}\omega\Big(\frac{1}{n-p+1}\Big)\bigg)\Bigg),
\end{equation}
where ${\bf K}(\rho )$ is the complete elliptic integral of the first kind, $\gamma (p)$ is
defined by $(\ref{28.09.11-11:12:43}),$ $\varepsilon _{n-p+1}=\sup\limits_{k\geqslant
n-p+1}{\big|\frac{ \psi (k+1)}{\psi (k)}-q\big|},$ and $O(1)$ is the quantity uniformly
bounded in $n,$ $p,$ $q$, $\omega $ and $\beta$.
\end{thrm}

Examples of convex moduli of continuity $\omega (t)$ satisfying condition
(\ref{28.09.11-11:14:59}) are the functions $\omega (t)=t^\alpha ,$ $\alpha \in(0,1),$
$\omega (t)=\ln^\beta (t+1),$ $\beta \in(0,1)$ and others. If $\omega (t)=t^\alpha ,$
$\alpha \in(0,1),$ the class $H_\omega $ turns into the well-known H\"{o}lder class
$H^\alpha .$ In this case equality (\ref{28.09.11-11:28:32}) has the form:
    $$\mathcal{E}(C^\psi _\beta H^\alpha ; V_{n,p})=\frac{
    \psi (n-p+1)}{p(n-p+1)^\alpha }\Bigg(\frac{2^{2+\alpha }}{\pi ^2}\frac{
    1-q^{2p}}{1-q^2}\textbf{K}(q^p)\int\limits_{
    0}^{\pi /2}t^\alpha \sin t\,dt$$
    $$+O(1)\bigg(
    \frac{1}{(1-q)^{\gamma (p)}(n-p+1)^{1-\alpha }}+\frac{
    \varepsilon _{n-p+1}}{(1-q)^2}\min\Big\{p, \frac{1}{1-q}\Big\}\bigg)\Bigg),\ \ n-p\to\infty.$$

Along with the Poisson kernel $P_{q,\beta }(t)$, the important examples of the kernels
$\Psi_{\beta }(t)$ (see (\ref{21.06.11-14:33:05})) whose coefficients $\psi (k)$ belong to
$D_q$, $q\in(0,1)$, are the Neumann kernel
    \begin{equation}\label{24.06.11-15:53:02}
    N_{q,\beta }(t)=\sum_{k=1}^{\infty}\frac{q^k}{k}\cos\Big(kt-\frac{\beta \pi }{2}\Big),
    \ \ q\in(0,1), \ \ \beta \in\mathbb{R}
    \end{equation}
and the polyharmonic Poisson kernel \cite[p. 256, 257]{Timan 2009}
    \begin{equation}\label{24.06.11-15:53:12}
    P_{q,\beta }(m,t)=\sum\limits_{
    k=1}^{\infty}\psi _{m}(k)\cos\Big(kt-\frac{\beta \pi }{2}\Big),\ \ \beta \in\mathbb{R},
    \end{equation}
where
$$\psi _m(k)=q^k\bigg(1+\sum\limits_{
j=1}^{m-1}\frac{(1-q^2)^j}{j!2^j}\prod_{l=0}^{j-1}(k+2l)\bigg),\ \  m\in\mathbb{N},\ \
q\in(0,1).$$

It's easy to verify that for the coefficients $\psi (k)=\frac{q^k}{k}$ of the Neumann kernel
$N_{q,\beta }(t)$ the equality
    \begin{equation}\label{21.06.11-15:14:29}
    \varepsilon _{n-p+1}=\sup\limits_{k\geqslant n-p+1}{\frac{q}{k+1}}=\frac{q}{n-p+2},
    \end{equation}
holds. As shown in \cite[p. 180]{Serdyuk_Sokolenko_BUL} (see, also, \cite[p.
132]{Serdyuk_Chaichenko}), in the case where $\psi (k)$ are the coefficients $\psi _m(k)$ of
the polyharmonic Poisson kernel $P_{q,\beta }(m,t)$,
    \begin{equation}\label{21.06.11-15:14:39}
    \varepsilon _{n-p+1}\leqslant \frac{(2m-3)q}{n-p+1},\ \ m=2,3,\ldots
    \end{equation}
(if $m=1$, then $\psi _m(k)=\psi _1(k)=q^k$ and so $\varepsilon_{n-p+1}=0$).

Thus from Theorem \ref{t2} and Theorem \ref{t3} we obtain the next assertions.

\begin{cor}\label{c1} Let $C^\psi _{\beta,s}$, $1\leqslant s\leqslant
\infty,$ and $C^\psi _\beta H_\omega $ be the classes generated by the coefficients $\psi
(k)=q^k/k$ of the Neumann kernel $N_{q,\beta }(t)$, $n,p\in\mathbb{N},$ $p\leqslant n,$ and
a convex modulus of continuity $\omega (t)$ satisfies condition $(\ref{28.09.11-11:14:59}).$
Then the following asymptotic equalities hold as $n-p\to\infty$
$$
    \mathcal{E}(C^\psi _{\beta,s};V_{n,p})=
    \frac{q^{n-p+1}}{p(n-p+1)}\Bigg(\frac{
    \|\cos t\|_{s'}}{\pi ^{1+1/s'}}K_{q,p}(s')$$
    $$+\frac{O(1)}{(n-p+1)}\bigg(
    \frac{1}{(1-q)^{\sigma
    (s',p)}}+\frac{q}{(1-q)^2}\min\Big\{p,\frac{1}{1-q}\Big\}\bigg)\Bigg),
    $$
    $$
    \mathcal{E}(C^\psi_{\beta}H_\omega ;V_{n,p})=
    \frac{q^{n-p+1}}{p(n-p+1)}\Bigg(
    \frac{4}{\pi ^2}\frac{1-q^{2p}}{1-q^2}{\bf K}(q^p)\int\limits_{
    0}^{\pi /2}\omega \left(\frac{2t}{n-p+1}\right)\sin t\,dt
    $$
    $$+\frac{O(1)}{n-p+1}\bigg(\frac{\omega (\pi )}{(1-q)^{\gamma(p)}}+
    \frac{q}{(1-q)^2}\min\Big\{p,
    \frac{1}{1-q}\Big\}\omega\Big(\frac{1}{n-p+1}\Big)\bigg)\Bigg),$$
where $K_{q,p}(s')$, $\sigma(s',p)$ and $\gamma(p)$ are defined by
 $(\ref{21.06.11-13:57:09})$, $(\ref{21.06.11-13:57:22})$ and $(\ref{28.09.11-11:12:43}),$
respectively, $s'=\frac{s}{s-1},$ and the quantities $O(1)$ are uniformly bounded in $n,$
$p$, $q$, $s$, $\omega $ and $\beta.$
\end{cor}

\begin{cor}\label{c2} Let $C^\psi _{\beta,s}$, $1\leqslant s\leqslant
\infty,$ and $C^\psi _\beta H_\omega $ be the classes generated by the coefficients $\psi
(k)=\psi _m(k)$ of the polyharmonic Poisson kernel $P_{q,\beta }(m,t),$ $m\in \mathbb{N},$
$n,p\in\mathbb{N},$ $p\leqslant n,$ and a convex modulus of continuity $\omega (t)$
satisfies condition $(\ref{28.09.11-11:14:59}).$ Then the following asymptotic equalities
hold as $n-p\to\infty$
$$
    \mathcal{E}(C^\psi _{\beta,s};V_{n,p})=$$
    $$=
    \frac{q^{n-p+1}}{p}\bigg(1+\sum\limits_{
j=1}^{m-1}\frac{(1-q^2)^j}{j!2^j}\prod_{l=0}^{j-1}(n-p+1+2l)\bigg)\Bigg(\frac{
    \|\cos t\|_{s'}}{\pi ^{1+1/s'}}K_{q,p}(s')$$
    $$+\frac{O(1)}{(n-p+1)}\bigg(
    \frac{1}{(1-q)^{\sigma
    (s',p)}}+\frac{mq}{(1-q)^2}\min\Big\{p,
    \frac{1}{1-q}\Big\}\bigg)\Bigg),
    $$
    $$
    \mathcal{E}(C^\psi_{\beta}H_\omega ;V_{n,p})=
    \frac{q^{n-p+1}}{p}\bigg(1+\sum\limits_{
j=1}^{m-1}\frac{(1-q^2)^j}{j!2^j}\prod_{l=0}^{j-1}(n-p+1+2l)\bigg)$$
    $$\times\Bigg(
    \frac{4}{\pi ^2}\frac{1-q^{2p}}{1-q^2}{\bf K}(q^p)\int\limits_{
    0}^{\pi /2}\omega \left(\frac{2t}{n-p+1}\right)\sin t\,dt
    $$
    $$+\frac{O(1)}{n-p+1}\bigg(\frac{\omega (\pi )}{(1-q)^{\gamma(p)}}+
    \frac{mq}{(1-q)^2}\min\Big\{p,
    \frac{1}{1-q}\Big\}\omega\Big(\frac{1}{n-p+1}\Big)\bigg)\Bigg),$$
where $K_{q,p}(s')$, $\sigma(s',p)$ and $\gamma(p)$ are defined by
 $(\ref{21.06.11-13:57:09})$, $(\ref{21.06.11-13:57:22})$ and $(\ref{28.09.11-11:12:43}),$
respectively, $s'=\frac{s}{s-1},$ and the quantities $O(1)$ are uniformly bounded in $n,$
$p$, $q$, $m,$ $s$, $\omega $ and $\beta.$
\end{cor}

\emph{Contact information}: \href{http://www.imath.kiev.ua/~funct}{Department of the Theory
of Functions}, Institute of Mathematics of Ukrainian National Academy of Sciences, 3,
Tereshenkivska st., 01601, Kyiv, Ukraine \vskip 0.2 cm

\href{mailto:serdyuk@imath.kiev.ua}{serdyuk@imath.kiev.ua},
\href{mailto:ievgen.ovsii@gmail.com}{ievgen.ovsii@gmail.com},
\href{mailto:andreymap@rambler.ru}{andreymap@rambler.ru}

\end{document}